\documentclass[oneside,10pt]{article}          
\usepackage[b5paper]{geometry}	    
\usepackage{amsfonts,amsmath,latexsym,amssymb} 
\usepackage{amsthm}                
\usepackage{mathrsfs,upref}         
\usepackage{mathptmx}		    
\usepackage{mia}	            

\usepackage[active]{srcltx}

\usepackage{graphicx}

\usepackage{enumerate}

\usepackage[final]
{showkeys}
\usepackage[hyphens]{url}

\RequirePackage[OT1]{fontenc}

\usepackage{tikz}
\usetikzlibrary{decorations}
\usetikzlibrary{decorations.pathreplacing}

\RequirePackage[abspath,realmainfile]{currfile}
\usepackage{datetime} 
\usepackage{xcolor}

\usepackage{hyperref}

\newtheorem{theorem}{Theorem}           
\newtheorem{lemma}[theorem]{Lemma}               
\newtheorem{corollary}[theorem]{Corollary}

\theoremstyle{definition}

\newtheorem{remark}[theorem]{Remark}

\providecommand{\url}[1]{#1}

\renewcommand{\le}{\leqslant}
\renewcommand{\ge}{\geqslant}

\renewcommand{\O}{\underset{\ffrown}{<}}

\newcommand{\lo}{\mathsf{lo}} 
\newcommand{\hi}{\mathsf{up}}

\renewcommand{\AA}{\mathsf{A}}
\renewcommand{\O}{\mathsf{O}}

\renewcommand{\L}{\mathsf{L}}

\newcommand{\R}{\mathbb{R}}                  
\newcommand{\N}{\mathbb{N}}                 
\newcommand{\Z}{\mathbb{Z}}

\newcommand{\ga}{\gamma}
\newcommand{\Ga}{\Gamma}
\newcommand{\de}{\delta}
\newcommand{\De}{\Delta}

\newcommand{\ka}{\kappa}

\newcommand{\sign}{\operatorname{sign}}

\renewcommand{\ka}{\lceil a\rceil}

\newcommand{\widesim}[2][1.5]{
  \mathrel{\underset{#2}{\scalebox{#1}[1]{$\sim$}}}
}

\newtheorem*{theorem*}{Theorem~\ref{th:}'}
\newtheorem*{thm*}{Theorem~\ref{th:der}'}

\newtheorem*{corollary*}{Corollary~\ref{cor:A}'}
\newtheorem*{cor*}{Corollary~\ref{cor:Ader}'}
\newtheorem{proposition}[theorem]{Proposition}
\newtheorem*{propA}{Proposition~A}
\theoremstyle{remark}

\numberwithin{equation}{section}
\numberwithin{theorem}{section}
 
\begin{document}

\title[Exact lower and upper bounds on the incomplete gamma function]{Exact lower and upper bounds on the incomplete gamma function}


\author{Iosif Pinelis}

\address{Iosif Pinelis, Department of Mathematical Sciences\\
Michigan Technological University\\
Hough\-ton, Michigan 49931, USA
\email{ipinelis@mtu.edu}}

\CorrespondingAuthor{Iosif Pinelis}


\date{08.25.2019
}                               

\keywords{Incomplete gamma function, exact bounds, inequalities, Gautschi inequalities
}

\subjclass{
33B20, 26D07, 26D15 
}

        

\begin{abstract}
Lower and upper bounds $B_a(x)$ on the incomplete gamma function $\Ga(a,x)$ are given for all real $a$ and all real $x>0$. These bounds $B_a(x)$ are exact in the sense that $B_a(x)\widesim{x\downarrow0}\Ga(a,x)$ and $B_a(x)\widesim{x\to\infty}\Ga(a,x)$. 
Moreover, the relative errors of these bounds are rather small for other values of $x$, away from $0$ and $\infty$. 
\end{abstract}

\maketitle



\tableofcontents

\section{Statements of main results}
\label{intro}


Take any real $a$ and any real $x>0$. The corresponding value of the incomplete gamma function is given by the formula 
\begin{equation}\label{eq:f}
	\Ga(a,x)=\int_x^\infty t^{a-1} e^{-t}\,dt. 
\end{equation}
Let 
\begin{equation}\label{eq:b}
	b_a:=\left\{
	\begin{alignedat}{2}
	&\Ga(a+1)^{1/(a-1)}&&\text{\quad if }a\in(-1,\infty)\setminus\{1\},\\
	&
	e^{1-\ga} &&\text{\quad if }a=1,   
	\end{alignedat}
	\right.
\end{equation}
where $\ga=0.577\ldots$ is the Euler constant. 

One may note here that $b_a>0$ for all $a>-1$. The value of $b_a$ at $a=1$ is defined in \eqref{eq:b} by continuity (see Lemma~\ref{prop:b} and its proof for details). 


Consider next 
\begin{equation}\label{eq:g}
G_a(x):=\left\{
	\begin{alignedat}{2}
	&x^{-2}\,e^{-x} &&\text{\quad if }a=-1, \\ 
	&\frac{(x+b_a)^a-x^a}{ab_a}\,e^{-x} &&\text{\quad if }a\in(-1,\infty)\setminus\{0\}, \\  
	&e^{-x}\ln\frac{x+1}x &&\text{\quad if }a=0.   
	\end{alignedat}	
		\right.
\end{equation}
One may note here that $G_a(x)$ is continuous in $a\ge-1$ for each $x>0$. 

Further, introduce 
\begin{equation}\label{eq:g-rev}
g_a(x):=\Big(\frac{(x+2)^a-x^a-2^a}{2a}+\Ga(a)\Big)\,e^{-x} 
\end{equation}
for $a>0$. 

Let us say that a bound $B_a(x)$ on $\Ga(a,x)$ is exact at $x=0$ if $B_a(x)\widesim{x\downarrow0}\Ga(a,x)$; similarly defined is the exactness at $x=\infty$.
As usual, we write $u\sim v$ for $u/v\to1$. 

\begin{theorem}\label{th:}
Take any real $a\ge-1$. Then (for all real $x>0$)  
\begin{alignat}{2}
		&\Ga(a,x)<G_a(x)&&\text{\quad if\quad }-1\le a<1, \label{eq:-1 le a<1}\\
		&g_a(x)=G_a(x)=\Ga(a,x)=e^{-x}&&\text{\quad if\quad }a=1, \label{eq:a=1} \\
		&g_a(x)<G_a(x)<\Ga(a,x)&&\text{\quad if\quad }1<a<2, \label{eq:1<a<2}\\
		&g_a(x)=G_a(x)=\Ga(a,x)=e^{-x}(1+x)&&\text{\quad if\quad }a=2, \notag \\
		&\Ga(a,x)<g_a(x)<G_a(x)&&\text{\quad if\quad }2<a<3, \notag \\
		&\Ga(a,x)=g_a(x)=e^{-x}(2 + 2 x + x^2)<G_a(x)&&\text{\quad if\quad }a=3, \notag \\
		&g_a(x)<\Ga(a,x)<G_a(x)&&\text{\quad if\quad }a>3.	\label{eq:a>3}
\end{alignat}

Also, for each real $a\ge0$ the bound $G_a(x)$ on $\Ga(a,x)$ is exact both at $x=0$ and at $x=\infty$. 
Further, for each real $a\ge1$ the bound $g_a(x)$ on $\Ga(a,x)$ is exact both at $x=0$ and at $x=\infty$. 
Moreover, the bound $G_a(x)$ is exact at $x=\infty$ for each real $a\ge-1$. 

\end{theorem}

Thus, for $a>3$ the bounds $G_a(x)$ and $g_a(x)$ on $\Ga(a,x)$ bracket $\Ga(a,x)$ from above and below, respectively. 

In the simple cases $a=1$ and $a=2$, the bounds $G_a(x)$ and $g_a(x)$ on $\Ga(a,x)$ are exact in sense that they coincide with $\Ga(a,x)$. In the same sense, the bound $g_a(x)$ on $\Ga(a,x)$ is exact if $a=3$, in contrast with the bound $G_a(x)$. 

If $1<a<2$, then the lower bound $G_a(x)$ on $\Ga(a,x)$ is better (that is, closer to $\Ga(a,x)$) than the lower bound $g_a(x)$. If $2<a\le3$, then, vice versa, the upper bound $g_a(x)$ on $\Ga(a,x)$ is better 
than the upper bound $G_a(x)$. 

Theorem~\ref{th:}, which concerns the case $a\ge-1$, is complemented by the following result. 

\begin{theorem}\label{th:a<-1}
Take any real $a<-1$. Then (for all real $x>0$) 
\begin{equation}\label{eq:bnds,a<-1}
		g_a^\lo(x)<\Ga(a,x)<g_a^\hi(x), 
\end{equation}
where 
\begin{equation*}
	g_a^\lo(x):=\frac{ x^a e^{-x} (x-a-1)}{(x-a)^2+a}\quad\text{and}\quad 
	g_a^\hi(x):=\frac{ x^a e^{-x}}{x-a}. 
\end{equation*}


Also, each of the bounds $g_a^\lo(x)$ and $g_a^\hi(x)$ on $\Ga(a,x)$ is exact 
both at $x=0$ and at $x=\infty$. 

Actually, the statements in this theorem concerning $g_a^\hi(x)$ hold for all $a<0$.  
\end{theorem}


The bounds on $\Ga(a,x)$ presented in Theorems~\ref{th:} and \ref{th:a<-1} are rather simple and appear natural. In particular, we shall see in Section~\ref{proofs} that the different pieces in the proofs of these bounds fit together 
tightly.

\section{Discussion}
\label{discuss}

Another nice asymptotic exactness property of the bracketing bounds $g_a^\lo(x)$ and $g_a^\hi(x)$ on $\Ga(a,x)$ is as follows. 

\begin{proposition}\label{prop:a->-infty,unif}
\begin{equation}\label{eq:a->-infty,unif}
g_a^\lo(x)\widesim{a\to-\infty}\Ga(a,x)\widesim{a\to-\infty}g_a^\hi(x)\quad\text{uniformly in $x>0$. }   
\end{equation}
Equivalently, 
\begin{equation}\label{eq:max ghi/glo}
	\max_{x>0}\dfrac{g_a^\hi(x)}{g_a^\lo(x)}\underset{a\to-\infty}\longrightarrow1. 
\end{equation}
\end{proposition}

In contrast with \eqref{eq:a->-infty,unif}--\eqref{eq:max ghi/glo}, bounds 
$g_a(x)$ and $G_a(x)$ on $\Ga(a,x)$ in \eqref{eq:a>3} -- which bracket $\Ga(a,x)$  for $a\ge3$ -- exhibit the following explosion phenomenon: 

\begin{proposition}\label{prop:a->infty}
\begin{equation}\label{eq:max ghi/glo,infty}
	\max_{x>0}\dfrac{G_a(x)}{\Ga(a,x)}\underset{a\to\infty}\longrightarrow\infty
	\quad\text{and}\quad
	\max_{x>0}\dfrac{\Ga(a,x)}{g_a(x)}\underset{a\to\infty}\longrightarrow\infty.  
\end{equation}
\end{proposition}



\medskip
\hrule
\medskip

One can find quite a few bounds on the incomplete gamma function in the literature, including papers
\cite{gautschi59,temme79,alzer97,qi99,nat-pal00,paris02,qi02,laforgia-nat06,borwein-chan,alzer-baricz12,neuman13,yang-zhang-chu17,greengard-rokhlin19}. 

A distinctive feature of our bounds on $\Ga(a,x)$ is their exactness both at $x=0$ and at $x=\infty$. It appears that this feature can be found only in few other papers. 

Apparently the first of them was the paper by Gautschi \cite{gautschi59}, containing the inequalities 
\begin{equation}\label{eq:gautschi}
	H(p,1/2,v)<e^{v^p}\int_v^\infty e^{-u^p}\,du
	\le H(p,c_p,v)  
\end{equation}
for real $p>1$ and real $v>0$, 
where $H(p,c,v):=c\big((v^p+1/c)^{1/p}-v\big)$ and $c_p:=\Ga(1+1/p)^{p/(p-1)}$. 

As noted in \cite{gautschi59}, it is easy to rewrite inequalities \eqref{eq:gautschi} in terms of the incomplete gamma function. Indeed, using the substitutions $p=1/a$, $v=x^{1/p}=x^a$, and $u=t^{1/p}=t^a$, we see that the second inequality in \eqref{eq:gautschi} (for $p>1$) becomes the (non-strict version of the) case of inequality \eqref{eq:-1 le a<1} for $a\in(0,1)$. The limit case $p=\infty$ of the second inequality in \eqref{eq:gautschi} similarly corresponds to the case $a=0$ of inequality \eqref{eq:-1 le a<1}. 
Thus, the second inequality in \eqref{eq:gautschi} can be considered a special case of \eqref{eq:-1 le a<1}, and it is therefore exact at $x=0$ and at $x=\infty$ -- or, in terms of \eqref{eq:gautschi}, at $v=0$ and at $v=\infty$. 

However, it is easy to see that the lower bound $H(p,1/2,v)$ on $e^{v^p}\int_v^\infty e^{-u^p}\,du$ in \eqref{eq:gautschi} is exact only at $v=\infty$, but not at $v=0$. 
The bound $g_a(x)$, defined in \eqref{eq:g-rev}, can then be viewed as a ``corrected'' version of $H(p,1/2,v)$ that is exact, for appropriate values of $a$, both at $x=0$ and at $x=\infty$. 

It was pointed out in the review of the paper \cite{gautschi59} in Mathematical Reviews \cite{MR-gautschi59} that ``As it stands, the proof is only valid if $p$ is an integer, but, in a correction, the author has indicated a modification which validates it for all $p>1$.'' Apparently \cite{gaut-personal}, no proof of \eqref{eq:gautschi} for the values $p\in(1,\infty)\setminus\Z$ -- which correspond to $a\in(0,1)\setminus\{\frac12,\frac13,\frac14,\dots\}$ -- has so far been published. 

Gautschi's result was complemented in \cite{yang-zhang-chu17}, where it was shown that $\Ga(a,x)>G_a(x)$ for $a\in(1,2)$ and $\Ga(a,x)<G_a(x)$ for $a>2$ (again, with $x>0$); cf.\ Theorem~\ref{th:} of the present paper. The method used in \cite{yang-zhang-chu17} was based on results in \cite{pin06}, restated, however, in terms of the function $H_{f,g}:=\frac{f'}{g'}\,g-f$, which differs only by the sign factor $\sign(g')$ from the function $\tilde\rho$ introduced and used in \cite{pin06}.

\medskip
\hrule
\medskip

\begin{remark}\label{rem:omissions}
For ``most'' real values of $a$, Theorems~\ref{th:} and \ref{th:a<-1} taken together provide both lower and upper bounds on $\Ga(a,x)$, each of which is exact both at $x=0$ and at $x=\infty$. However, there are a few gaps in the coverage by Theorems~\ref{th:} and \ref{th:a<-1}: 
\begin{enumerate}[{Gap} 1:]
	\item the absence of a lower bound on $\Ga(a,x)$ for $a\in[-1,1)$;
	\item the absence of an upper bound on $\Ga(a,x)$ for $a\in(1,2)$;
	\item the absence of a lower bound on $\Ga(a,x)$ for $a\in(2,3)$. 	
\end{enumerate}
Moreover, we shall address 
\begin{enumerate}[{Gap} 4:]
	\item for $a\in[-1,0)$, the upper bound $G_a(x)$ on $\Ga(a,x)$ is exact only at  $x=\infty$, but not at $x=0$. 	
\end{enumerate}
\qed
\end{remark}


To fill these gaps, and to address the explosion phenomenon presented in Proposition~\ref{prop:a->infty}, we can use the following shift technique. 

%
%

Integrating by parts, we have  
\begin{equation}\label{eq:Ga,fwd}
	\Ga(a,x)=x^{a-1} e^{-x}+(a-1)\Ga(a-1,x)
\end{equation}
for all real $a$ and all $x>0$. Iterating this recursion in $a$, we see that  
\begin{equation}\label{eq:sum}
	\Ga(a,x)=x^{a-1} e^{-x}\,\sum_{j=0}^{k-1}(a-1)_j\, x^{-j}+(a-1)_k\, \Ga(a-k,x)
\end{equation}
for all natural $k
$ and all real $x>0$, where $(u)_j:=\prod_{i=0}^{j-1}(u-i)$, the $j$th falling factorial of $u$. 

Replacing now $\Ga(a-k,x)$ on the right-hand side of \eqref{eq:sum} by a bound $B_{a-k}(x)$ on $\Ga(a-k,x)$, we obtain the new, modified bound
\begin{equation}\label{eq:B}
	B_{a;k}(x):=x^{a-1} e^{-x}\,\sum_{j=0}^{k-1}(a-1)_j\, x^{-j}+(a-1)_k\, B_{a-k}(x)  
\end{equation}
on $\Ga(a,x)$, which may be thought of as the (forward) $k$-shift of the bound $B_{a-k}(x)$ on $\Ga(a-k,x)$. 
In particular, if $(a-1)_k\ge0$, then this forward $k$-shift will transform a lower (respectively, upper) bound $B_{a-k}(x)$ on $\Ga(a-k,x)$ into a lower (respectively, upper) bound $B_{a;k}(x)$ on $\Ga(a,x)$. Similarly, if $(a-1)_k\le0$, then the forward $k$-shift will transform a lower (respectively, upper) bound into an upper (respectively, lower) one. 


For any bound $B_a(x)$ on $\Ga(a,x)$, consider the corresponding (signed) error and (signed) relative error of the approximation of $\Ga(a,x)$ by the bound $B_a(x)$: 
\begin{equation*}
	\De B_a(x):=B_a(x)-\Ga(a,x)\quad\text{and}\quad \de B_a(x):=\frac{\De B_a(x)}{\Ga(a,x)}. 
\end{equation*}
It is then obvious from \eqref{eq:sum} and \eqref{eq:B} that 
\begin{equation}\label{eq:De}
	\De B_{a;k}(x)=(a-1)_k\, \De B_{a-k}(x). 
\end{equation}
 
Now we obtain the following simple ``relative-error-taming" 

\begin{proposition}\label{prop:taming}
Take any natural $k$, any real $a\ge k$, and any real $x>0$. 
\begin{enumerate}[(i)]
	\item If $B_{a-k}(x)$ is a lower bound on $\Ga(a-k,x)$, then $B_{a;k}(x)$ is a lower bound on $\Ga(a,x)$, and $\de B_{a-k}(x)\le\de B_{a;k}(x)\le0$. 
	\item If $B_{a-k}(x)$ is an upper bound on $\Ga(a-k,x)$, then $B_{a;k}(x)$ is also an upper bound on $\Ga(a,x)$, and $0\le\de B_{a;k}(x)\le\de B_{a-k}(x)$. 
\end{enumerate} 
\end{proposition}

This follows immediately from \eqref{eq:De} and the inequality $$\Ga(a,x)\ge(a-1)_k\, \Ga(a-k,x);$$ in turn, the latter inequality follows immediately, in the conditions of Proposition~\ref{prop:taming}, from identity \eqref{eq:sum}.  

So, if $a\ge k\in\N$, then the forward $k$-shift can only reduce the absolute value of the relative error of a bound $B_a(x)$ on $\Ga(a,x)$. Immediately from Theorem~\ref{th:} and Proposition~\ref{prop:taming}, we obtain 

\begin{corollary}\label{prop:shift}
Take any real $a$ and any real $x>0$. 
\begin{enumerate}[(i)]
	\item If $a>2$ and $k=\ka-2$, then $1<a-k\le2$ and $\de G_{a-k}(x)\le\de G_{a;k}(x)\le0$. 
	\item If $a>3$ and $k=\ka-3$, then $2<a-k\le3$ and $0\le\de g_{a;k}(x)\le\de g_{a-k}(x)$.
	\item If $a>4$ and $k=\ka-4$, then $3<a-k\le4$ and 
	$\de g_{a-k}(x)\le\de g_{a;k}(x)\le0\le\de G_{a;k}(x)\le\de G_{a-k}(x)$.
\end{enumerate} 
\end{corollary}

Before stating the following proposition, let us note that $\de B_a(x)>-1$ whenever $B_a(x)>0$. 

\begin{proposition}\label{prop:bounded}
Take any real $a_*$. Then 
\begin{enumerate}[(i)]
	\item $\de G_a(x)$ is bounded away from $-1$ and $\infty$ over all $(a,x)\in[1,a_*]\times(0,\infty)$; 
	\item $\de g_a(x)$ is bounded away from $-1$ and $\infty$ over all $(a,x)\in[2,a_*]\times(0,\infty)$. 
\end{enumerate} 
\end{proposition}

We see that, in particular, Corollary~\ref{prop:shift} and Proposition~\ref{prop:bounded}, taken together, provide 
a lower bound and an upper bound on $\Ga(a,x)$ with relative errors bounded away from $-1$ and $\infty$ uniformly over all $(a,x)\in[2,\infty)\times(0,\infty)$. 
This fully addresses the explosion phenomenon described in Proposition~\ref{prop:a->infty}. 
Of course, the trade-off when using the shifted, better bounds $G_{a;k}(x)$ and $g_{a;k}(x)$ for large $k$ is that they are more complicated than the ``original" bounds  $G_a(x)$ and $g_a(x)$.  

The following proposition provides simple, if not very precise, bounds on $\Ga(a,x)$, to be used in the proof of Proposition~\ref{prop:bounded}, which is clearly of a qualitative nature. 

\begin{proposition}\label{prop:f_a}
Take any real $a\ge1$. Then $\Ga(a,x)\ge x^{a-1}e^{-x}$ for all real $x>0$ and $\Ga(a,x)\le x^{a-1}e^{-x}/\big(1-(a-1)/x\big)$ for all real $x>a-1$. 
\end{proposition}

The shift technique also allows us to fill the gaps described in Remark~\ref{rem:omissions}. 
Along with the forward shift described above, here we can use the corresponding backward shift. To obtain such a shift, let us begin by 
rewriting the forward-shift identity \eqref{eq:sum} in a ``backward'' manner: 
\begin{equation}\label{eq:back}
	\Ga(a,x)=\frac1{(a-1+k)_k}\Big(\Ga(a+k,x)-x^{a-1+k} e^{-x}\,\sum_{j=0}^{k-1}(a-1+k)_j\, x^{-j}\Big)
\end{equation} 
for all real $a$, all natural $k$ such that $(a-1+k)_k\ne0$, and all real $x>0$. Replacing here $\Ga(a+k,x)$ by a bound $B_{a+k}(x)$ on $\Ga(a+k,x)$, we obtain the  ``backward-shifted'' version of the bound $B_{a+k}(x)$: 
\begin{equation}\label{eq:B,back}
	B_{a;-k}(x):=\frac1{(a-1+k)_k}\Big(B_{a+k}(x)-x^{a-1+k} e^{-x}\,\sum_{j=0}^{k-1}(a-1+k)_j\, x^{-j}\Big).   
\end{equation}
In particular, if $(a-1+k)_k>0$, then this backward $k$-shift will transform a lower (respectively, upper) bound $B_{a+k}(x)$ on $\Ga(a+k,x)$ into a lower (respectively, upper) bound $B_{a;-k}(x)$ on $\Ga(a,x)$. Similarly, if $(a-1+k)_k<0$, then the backward $k$-shift will transform a lower (respectively, upper) bound into an upper (respectively, lower) one. 

Now we are ready to state the following propositions. 

\begin{proposition}\label{prop:g_{a;2}^lo} 
Take any real $a<1$ and recall \eqref{eq:B} and \eqref{eq:bnds,a<-1}. Then for the forward $2$-shift 
\begin{align}
	g_{a;2}^\lo(x)&=x^{a-1} e^{-x}\,\big(1+(a-1)/x)+(a-1)(a-2)\,g_{a-2}^\lo(x) \\ 
	&=\frac{e^{-x} x^a (x+3-a)}{x^2 + (4 - 2 a) x + (a - 1) (a - 2)} \label{eq:g_{a;2}^lo}
\end{align}
of the lower bound $g_{a-2}^\lo(x)$ on $\Ga(a-2,x)$ for all real $x>0$ we have 
\begin{equation}\label{eq:>g_{a;2}^lo}
	\Ga(a,x)>g_{a;2}^\lo(x), 
\end{equation}
so that $g_{a;2}^\lo(x)$ is a lower bound on $\Ga(a,x)$. 
Moreover, for each real $a<1$ the lower bound $g_{a;2}^\lo(x)$ on $\Ga(a,x)$ is exact at $x=\infty$. 
\end{proposition}

\begin{proposition}\label{prop:G_{a;-1}} 
Take any $a\in(-2,1)$, and recall \eqref{eq:B,back}, \eqref{eq:g}, and \eqref{eq:-1 le a<1}. If $a\ne0$, consider the backward $1$-shift 
\begin{align}
	G_{a;-1}(x)&=\tfrac1a\,\big(G_{a+1}(x) - x^a e^{-x}\big)  \label{eq:eq,G_{a;-1}}
\end{align}
of the bound $G_{a+1}(x)$ on $\Ga(a+1,x)$. Define $G_{0;-1}(x)$ by continuity: 
\begin{align}\label{eq:0,G_{a;-1}}
	G_{0;-1}(x)&:=\lim_{a\to0}G_{a;-1}(x) 
	=e^{-x} \Big[\Big(1+\frac{x}{b_1}\Big) \ln \Big(1+\frac{b_1}{x}\Big)-1\Big]. 
\end{align}
Then 
for all real $x>0$ 
\begin{equation}\label{eq:ineq,G_{a;-1}}
	\Ga(a,x)>G_{a;-1}(x), 
\end{equation}
so that $G_{a;-1}(x)$ is a lower bound on $\Ga(a,x)$. 
Moreover, for each $a\in(-2,1)$ the lower bound $G_{a;-1}(x)$ on $\Ga(a,x)$ is exact at $x=0$. 
\end{proposition}

\begin{proposition}\label{prop:G_{a;1}} 
Take any $a\in(1,3)$, and recall \eqref{eq:B}, \eqref{eq:g}, and Theorem~\ref{th:}. Then for the forward $1$-shift 
\begin{align*}
	G_{a;1}(x)&=x^{a-1} e^{-x}+(a-1)\,G_{a-1}(x) 
\end{align*}
of the bound $G_{a-1}(x)$ on $\Ga(a-1,x)$ for all real $x>0$ we have 
\begin{align}
	\Ga(a,x)<G_{a;1}(x)\quad &\text{if } 1<a<2, \label{eq:1<a<2,up} \\ 
\Ga(a,x)=	G_{a;1}(x)\quad &\text{if } a=2, \label{eq:a=2}\\  
	\Ga(a,x)>G_{a;1}(x)\quad &\text{if } 2<a<3. \label{eq:2<a<3,lo}
\end{align} 
Moreover, for each $a\in(1,3)$ the bound $G_{a;1}(x)$ on $\Ga(a,x)$ is exact both at $x=0$ and at $x=\infty$. 
\end{proposition}

\begin{proposition}\label{prop:g_{a;1}^lo} 
Take any real $a<0$. Then for the forward $1$-shift 
\begin{align}\label{eq:eq,g_{a;1}^lo}
	g_{a;1}^\lo(x)&=x^{a-1} e^{-x}+(a-1)\,g_{a-1}^\lo(x) 
	=\frac{e^{-x} x^a (1-a+x)}{(x-a)^2-a+2 x} 
\end{align}
of the lower bound $g_{a-1}^\lo(x)$ on $\Ga(a-1,x)$ for all real $x>0$ we have 
\begin{equation}\label{eq:ineq,g_{a;1}^lo}
	\Ga(a,x)<g_{a;1}^\lo(x)<g_a^\hi(x), 
\end{equation}
so that $g_{a;1}^\lo(x)$ is an upper bound on $\Ga(a,x)$, which is an improvement of the upper bound $g_a^\hi(x)$ on $\Ga(a,x)$. 
Moreover, for each real $a<0$ the upper bound $g_{a;1}^\lo(x)$ on $\Ga(a,x)$ is exact both at $x=0$ and at $x=\infty$. 
\end{proposition}

Propositions~\ref{prop:g_{a;2}^lo}--\ref{prop:g_{a;1}^lo} fill the four gaps listed in Remark~\ref{rem:omissions}. In particular, inequalities \eqref{eq:1<a<2,up} and  \eqref{eq:2<a<3,lo} in Proposition~\ref{prop:G_{a;1}} cover Gaps~2 and 3, respectively, whereas Proposition~\ref{prop:g_{a;1}^lo} covers Gap~4. 
Finally, Gap~1 is covered by the following immediate corollary of Propositions~\ref{prop:g_{a;2}^lo} and \ref{prop:G_{a;-1}}:

\begin{corollary}\label{cor:-1 le a<1}
Take any $a\in(-2,1)$. 
Then 
for all real $x>0$ 
\begin{equation*}
	\Ga(a,x)>h_a(x):=G_{a;-1}(x)\vee g_{a;2}^\lo(x), 
\end{equation*}
so that $h_a(x)$ is a lower bound on $\Ga(a,x)$. 
Moreover, for each $a\in(-2,1)$ the lower bound $h_a(x)$ on $\Ga(a,x)$ is exact both at $x=0$ and at $x=\infty$. 
\end{corollary}

The drawback of the bound $h_a(x)$ on $\Ga(a,x)$ in Corollary~\ref{cor:-1 le a<1} is that, in contrast with all the other bounds on $\Ga(a,x)$ given in this paper, the bound $h_a$ is not a real-analytic function, but rather the maximum of two real-analytic functions, $G_{a;-1}$ and $g_{a;2}^\lo$. 

\medskip
\hrule
\medskip

\begin{figure}[h]%
\includegraphics[width=1\columnwidth]{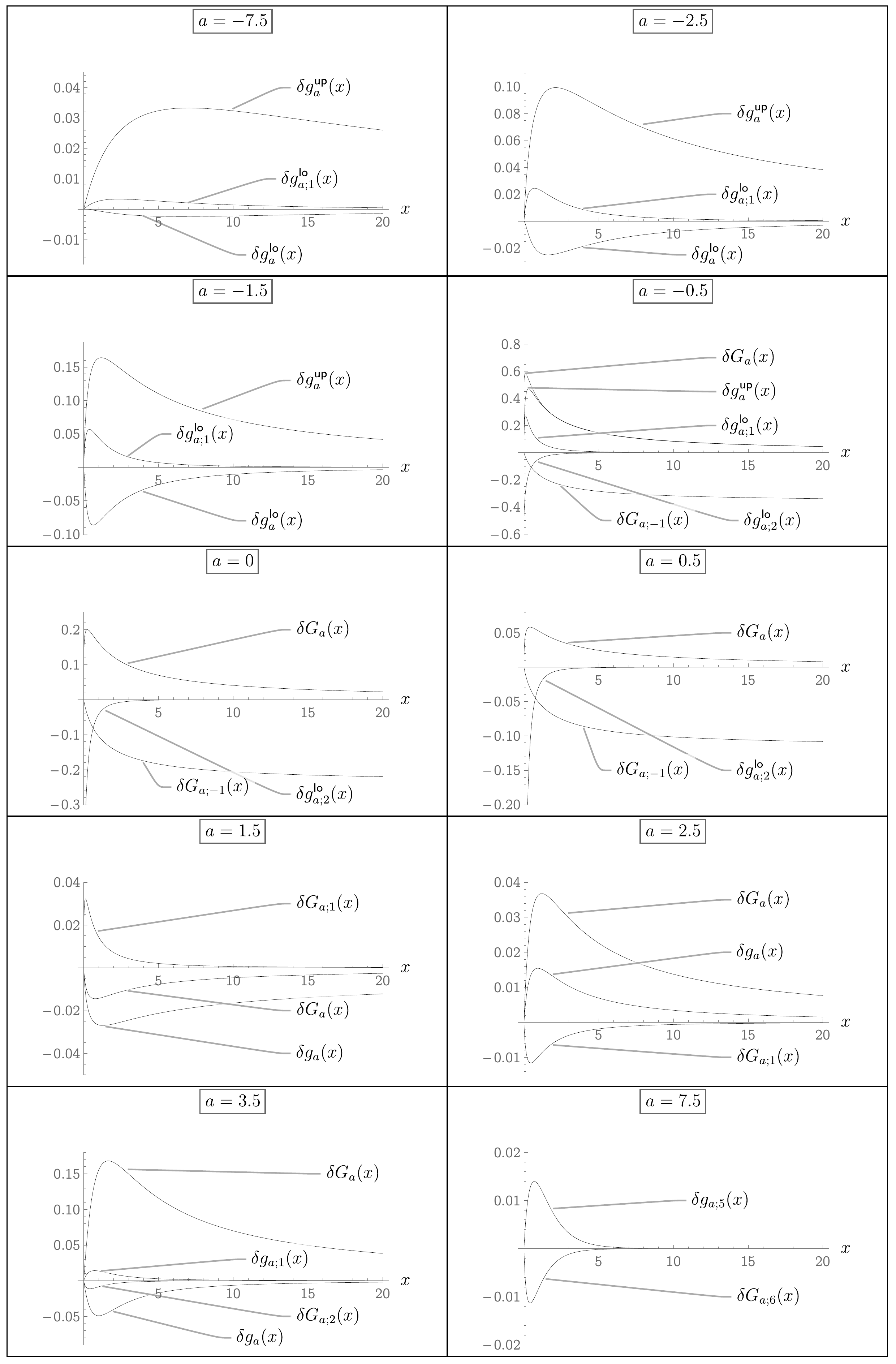}%
\caption{Graphs of signed relative errors of bounds on $\Ga(a,x)$.}%
\label{fig:grid}%
\end{figure}


%
%
%
%

Figure~\ref{fig:grid} shows graphs of the signed relative errors of various bounds on $\Ga(a,x)$ presented above for selected values of $a$, namely, for $a\in\{-7.5,-2.5,-1.5,-0.5,0,\break 
0.5,1.5,2.5,3.5, 7.5\}$. 



\section{Proofs}
\label{proofs}


The proofs are based mainly on the following ``
special-case l'Hospital-type rules for monotonicity'' given in \cite[Propositions~4.1 and 4.3]{pin06}: 

\begin{propA}\label{prop:lhosp}
Let $-\infty\le A<B\le\infty$. Let $f$ and $g$ be differentiable functions defined on the interval $(A,B)$ such that the functions $f$ and $f'$ do not take on the zero value
and do not change their respective signs on $(A,B)$. Suppose also that $f(A+)=g(A+)=0$ or $f(B-)=g(B-)=0$. 
Consider the ratio $r:=g/f$ and the ``derivative'' ratio $\rho:=g'/f'$. Then we have the following: 
\begin{enumerate}[(i)]
	\item If $\rho$ is increasing or decreasing on $(A,B)$, then $r$ is so as well, respectively. 
\item If $\rho$ is increasing-decreasing or decreasing-increasing on $(A,B)$, 
then $r$ is so as well, respectively.
\end{enumerate} 
\end{propA}

Here we say that a function $r$ on $(A,B)$ is increasing-decreasing if there is some point $C\in[A,B]$ such that $r$ is 
increasing on $(A,C)$ and 
decreasing on $(C,B)$. The term ``decreasing-increasing'' is defined similarly, so that $r$ is decreasing-increasing if and only if the function $-r$ is increasing-decreasing. In particular, if $r$ is increasing or decreasing on the entire interval $(A,B)$, then $r$ is both increasing-decreasing and decreasing-increasing on $(A,B)$. 

In this paper the terms ``increasing" and ``decreasing" are understood in the strict sense: namely, as ``strictly increasing" and ``strictly decreasing", respectively. 

General versions of this special l'Hospital-type rule for monotonicity, without the assumption that $f(A+)=g(A+)=0$ or $f(B-)=g(B-)=0$ are also known; see again \cite{pin06} and references therein, especially \cite{pin01}. 

Next, let us say that a function $h\colon(0,\infty)\to\R$ is strictly concave-convex if, for some $c\in(0,\infty)$, the function $h$ is strictly concave on $(0,c]$ and strictly convex on $[c,\infty)$. Let us say that $h$ is strictly convex-concave if $-h$ is strictly concave-convex. 

\begin{lemma}\label{lem:rho''} 
Let a function $h\colon(0,\infty)\to\R$ be such that $h(\infty-)\in\R$. Then, if $h$ is strictly concave-convex, then $h$ is increasing-decreasing; if $h$ is strictly convex-concave, then $h$ is decreasing-increasing.  
\end{lemma}

\begin{proof} Without loss of generality, $h$ is strictly convex-concave. Hence, $h$ is de\-creasing-increasing on $(0,c]$ and increasing-decreasing on $[c,\infty)$, for some $c\in(0,\infty)$. Moreover, if $h$ were decreasing on $[d,\infty)$ for some real $d\ge c$, then, because $h$ is strictly concave on $[d,\infty)$, we would have $h(\infty-)=-\infty$, which would contradict the condition $h(\infty-)\in\R$. 
Thus, $h$ is increasing on $[c,\infty)$ and decreasing-increasing on $(0,c]$, which implies that $h$ is decreasing-increasing on $(0,\infty)$.  

This completes the proof of Lemma~\ref{lem:rho''}. 
\end{proof}

\begin{proof}[Proof of Theorem~\ref{th:}]
This follows immediately from Propositions~\ref{prop:G}, \ref{prop:g}, and \ref{prop:compar} below. 
\end{proof}

\begin{proposition}\label{prop:G}
Take any real $a\ge-1$. Then (for all real $x>0$) 
\begin{equation}\label{eq:G bounds}
		\Ga(a,x)\left\{
	\begin{alignedat}{2}
	&<G_a(x)&&\text{\quad if }a\in[-1,1)\cup(2,\infty),\\
	&>G_a(x)&&\text{\quad if }a\in(1,2), \\ 	
	&=G_a(x)&&\text{\quad if }a\in\{1,2\}.	
	\end{alignedat}
	\right.
\end{equation}
Also, for each real $a\ge0$ the bound $G_a(x)$ on $\Ga(a,x)$ is exact both at $x=0$ and at $x=\infty$. In fact, this bound is exact both at $x=\infty$ for each real $a\ge-1$. 
\end{proposition}

\begin{proposition}\label{prop:g}
Take any real $a\ge1$. Then (for all real $x>0$) 
\begin{equation}\label{eq:g bounds}
		\Ga(a,x)\left\{
	\begin{alignedat}{2}
	&>g_a(x)&&\text{\quad if }a\in(1,2)\cup(3,\infty),\\
	&<g_a(x)&&\text{\quad if }a\in(2,3), \\ 	
	&=g_a(x)&&\text{\quad if }a\in\{1,2,3\}.	
	\end{alignedat}
	\right.
\end{equation}
Also, the bound $g_a(x)$ on $\Ga(a,x)$ is exact both at $x=0$ and at $x=\infty$. 
In fact, this bound is exact both at $x=\infty$ for each real $a\ge-1$. 
%
\end{proposition}

\begin{proposition}\label{prop:compar}
For all real $x>0$ 
\begin{alignat}{2}
		&g_a(x)<G_a(x)&&\text{\quad if\quad }a\in(1,2)\cup(2,\infty). \label{eq:g<G} 
\end{alignat}
\end{proposition}

To prove Proposition~\ref{prop:G}, we shall need the following two lemmas. 

\begin{lemma}\label{prop:b}
$b_a$ is continuously increasing in real $a>-1$ from $b_{(-1)+}=0$ to $b_0=1$ to $b_1=e^{1-\ga}$ to $b_2=2$ to $b_{\infty-}=\infty$.  
\end{lemma}

\begin{proof}
%

The most essential ingredient of this proof is Proposition~A, stated in the beginning of this section. 
Indeed, for $a\in(-1,\infty)\setminus\{1\}$ we have $\ln b_a=\frac{\ln\Ga(a+1)}{a-1}$, and the ``derivative'' ratio for the ratio $\frac{\ln\Ga(a+1)}{a-1}$ is $\frac d{da}\,\ln\Ga(a+1)$, which is increasing in $a\in(-1,\infty)$, since the function $\Ga$ is strictly log convex. So, by part~(i) of Proposition~A, $b_a$ is increasing in $a\in(-1,1)$ and in $a\in(1,\infty)$. 

Moreover, using the well known fact (see e.g.\ \cite[formula (1.2.12)]{andrews}) that 
\begin{equation}\label{eq:Ga'(1)}
\Ga'(1)=-\ga, 	
\end{equation}
the identity $\Ga(a+1)=a\Ga(a)$, and l'Hospital's rule, we see that $b_a$ is continuous in $a$ at $a=1$ and hence in all real $a>-1$. 
Therefore, $b_a$ is increasing in all real $a>-1$ . 

%

The equality $b_{(-1)+}=0$ follows immediately from the identity $\Ga(a+1)=\frac{\Ga(a+2)}{a+1}$ for $a\ne1$. The equalities $b_0=1$ and $b_2=2$ are trivial. Finally, the equality $b_{\infty-}=\infty$ follows easily from Stirling's formula. 
\end{proof}

\begin{lemma}\label{lem:} We have 
\begin{equation*}
	\begin{alignedat}{2}
	&a<b_a<2&&\text{\quad if }a\in
	(-1,2),\\
	&a>b_a>2&&\text{\quad if }a\in(2,\infty). 
	\end{alignedat}
\end{equation*}
\end{lemma}

\begin{proof} If $a\in(-1,0]$, then the inequality $a<b_a$ is obvious and the inequality $b_a<2$ follows by Lemma~\ref{prop:b}. 

Take now any real $a>0$. Let 
\begin{equation*}
	h(a):=(a-1)\ln(b_a/a)=\ln\Ga(a+1)+(1-a)\ln a. 
\end{equation*}
Then $h''(a)=\psi'(a+1)-\frac1a-\frac1{a^2}<\psi'(a+1)-\frac1a$, where, as usual, $\psi:=(\ln\Ga)'=\Ga'/\Ga$, and, by \cite[formula~(1.2.14)]{andrews}, 
\begin{equation*}
\psi'(a+1)=\sum_{k=0}^\infty\frac1{(a+1+k)^2}<\int_a^\infty\frac{dx}{x^2}=\frac1a. 
\end{equation*}
So, $h''<0$ and hence $h$ is strictly concave on $(0,\infty)$. Also, $h(1)=h(2)=0$. 
Hence, $h<0$ on $(0,1)\cup(2,\infty)$ and $h>0$ on $(1,2)$. 
Now the inequalities $a<b_a$ for $a\in(0,1)\cup(1,2)$ and $a>b_a$ for $a\in(2,\infty)$ follow immediately from the definition of $h$. 
The inequalities $b_a<2$ for $a\in(0,1)\cup(1,2)$ and $b_a>2$ for $a\in(2,\infty)$, as well as the inequalities $a<b_a<2$ for $a=1$, follow immediately from Lemma~\ref{prop:b}. 

This completes the proof of Lemma~\ref{lem:}. 
\end{proof}

The following two very simple lemmas will be used repeatedly. 

\begin{lemma}\label{lem:infty}
For each real $a$ we have $\Ga(a,x)\sim x^{a-1}e^{-x}$ as $x\to\infty$. 
\end{lemma}

This follows immediately by the l'Hospital rule.  

\begin{lemma}\label{lem:0}
Take any real $a$. Then  
\begin{equation}\label{eq:Ga sim}
		\Ga(a,x)\widesim{x\downarrow0}
		\left\{
	\begin{alignedat}{2}
	&-x^a/a&&\text{\quad if }a<0,\\
	&-\ln x&&\text{\quad if }a=0, \\ 	
	&\Ga(a)&&\text{\quad if }a>0.	
	\end{alignedat}
	\right.
\end{equation}
Moreover,
\begin{equation}\label{eq:Ga,a=0}
	\Ga(0,x)=-\ln x-\ga+O(x)  
\end{equation}
as $x\downarrow0$, where, again, $\ga$ is the Euler constant.
\end{lemma}

\begin{proof}
The first two asymptotic relations in \eqref{eq:Ga sim} follow immediately by the \break  l'Hospital rule; the third asymptotic relation in \eqref{eq:Ga sim} follows immediately by, say, the  dominated convergence theorem.  

To prove \eqref{eq:Ga,a=0}, use the identities $\int_0^\infty e^{-t}\ln t\,dt=\Ga'(1)$ and \eqref{eq:Ga'(1)} to write 
\begin{equation*}
	\int_0^\infty e^{-t}\ln t\,dt=-\ga. 
\end{equation*}
So, integration by parts yields 
\begin{multline*}
	\Ga(0,x)=\int_x^\infty\frac1t\,e^{-t}\,dt
	=-e^{-x}\ln x-\ga-\int_0^x e^{-t}\ln t\,dt \\ 
	=-(1-x)\ln x+O(x^2|\ln x|)-\ga-\int_0^x \ln t\,dt+O(x^2|\ln x|) \\ 
	=-\ln x-\ga+x+O(x^2|\ln x|)=-\ln x-\ga+O(x)  
\end{multline*}
as $x\downarrow0$, which proves \eqref{eq:Ga,a=0} as well. 
\end{proof}


\begin{proof}[Proof of Proposition~\ref{prop:G}]
The cases with $a\in\{1,2\}$ in \eqref{eq:G bounds} are straightforward. 

Take now any $a\in[-1,\infty)\setminus\{1,2\}$. 
By the mean value theorem, $G_a(x)\sim x^{a-1}e^{-x}$ as $x\to\infty$, and now the 
exactness of the bound $G_a(x)$ on $\Ga(a,x)$ at $x=\infty$ follows 
by Lemma~\ref{lem:infty}. 


%

%
%

The exactness of the bound $G_a(x)$ on $\Ga(a,x)$ at $x=0$ for each real $a\ge0$ follows immediately from \eqref{eq:Ga sim} and \eqref{eq:g}. 

It remains to prove the inequalities in \eqref{eq:G bounds}. 
This proof relies on Lemma~\ref{lem:} and the ``special-case l'Hospital-type rules for monotonicity'' cited in Proposition~A.  

We are going to apply Proposition~A to the functions $f=f_a$ and $g=G_a$, where 
\begin{equation}\label{eq:f_a}
	f_a(x):=\Ga(a,x)
\end{equation}
and $G_a$ defined by \eqref{eq:g}. Then 
for $a\in(-1,\infty)\setminus\{0,1,2\}$
\begin{equation}\label{eq:rho to 1}
	\rho(x)=\rho_a(x):=\frac{G'_a(x)}{f'_a(x)}
	=x\,\frac{(1+b_a/x)^a-1}{a b_a}
	-\frac{(1+b_a/x)^{a-1}-1}{b_a}\underset{x\to\infty}\longrightarrow1  
\end{equation}
and 
\begin{equation*}
	\rho''(x)x^{a + 1} (b_a + x)^{3 - a}=(a - 1) \big(b_a (b_a - a) - (2-b_a) x\big). 
\end{equation*}

Consider now the case $a\in(1,2)$. Then, by Lemma~\ref{lem:}, for 
$c:=c_a:=\break 
b_a (b_a - a)/(2-b_a)\in(0,\infty)$ we have $\rho''>0$ on the interval $(0,c)$, and $\rho''<0$ on the interval $(c,\infty)$. So, by Lemma~\ref{lem:rho''}, $\rho$ is decreasing-increasing on $(0,\infty)$. 
Also, $f(\infty-)=g(\infty-)=0$. 
Therefore, by part~(ii) of Proposition~A, $r=g/f$ is decreasing-increasing on $(0,\infty)$. 
Also, by 
the exactness of the bound $G_a(x)$ on $\Ga(a,x)$, we have $r(0+)=r(\infty-)=1$. 
It follows that $r<1$ on $(0,\infty)$, which means that the second inequality in \eqref{eq:G bounds} holds. 

The first inequality in \eqref{eq:G bounds} is proved quite similarly for $a\in(-1,0)\cup(0,1)\cup(2,\infty)$ -- except that for $a\in(-1,0)$ by Lemma~\ref{prop:b} we have $r(0+)=1/b_a>1$, rather than $r(0+)=1$. 

In the remaining cases $a=-1$ and $a=0$, the proof of the first inequality in \eqref{eq:G bounds} is similar and even easier, especially in the case $a=-1$, where $\rho(x)=1+2/x$ is obviously decreasing in $x>0$; in the case $a=0$, we have $\rho''(x)x (1 + x)^3=x-1$.


This completes the proof of Proposition~\ref{prop:G}. 
\end{proof}

\begin{proof}[Proof of Proposition~\ref{prop:g}] 
This proof is very similar to, and even a bit simpler than, the proof of Proposition~\ref{prop:G}. 

Indeed, the cases with $a\in\{1,2,3\}$ in \eqref{eq:g bounds} are straightforward. 

Take now any $a\in[1,\infty)\setminus\{1,2,3\}$. 
By the mean value theorem, $g_a(x)\sim x^{a-1}e^{-x}$ as $x\to\infty$, and now the 
exactness of the bound $g_a(x)$ at $x=\infty$ follows by 
Lemma~\ref{lem:infty}. 

For any $a>0$ (and hence for any $a>1$), we have the trivial equalities $\Ga(a,0+)=\Ga(a)=g_a(0+)\in(0,\infty)$, so that 
the bound $g_a(x)$ is exact at $x=0$. 

It remains to prove the inequalities in \eqref{eq:g bounds}. 
We are going to apply Proposition~A to the functions $f=f_a$ and $g=g_a$, with $f_a$ defined by \eqref{eq:f_a} 
and $g_a$ defined by \eqref{eq:g-rev}. Then 
\begin{equation}\label{eq:rho to 1,g}
	\rho(x)=\frac{g'(x)}{f'(x)}
	=x\,\frac{(1+2/x)^a-1}{2a}
	-\frac{(1+2/x)^{a-1}-1}2
	+\frac{C_a}{a}\,x^{1-a}
	\underset{x\to\infty}\longrightarrow1  
\end{equation}
and 
\begin{equation}
	\rho''(x)\frac{x^{a + 1}}{2(a - 1)}= (2 - a) (2 + x)^{a - 3} + C_a/2, \label{eq:rho'',g} 
\end{equation}
where 
\begin{equation}\label{eq:C_a}
	C_a:=\Ga(a + 1) - 2^{a - 1}. 
\end{equation}
Note that the ratio $\Ga(a + 1)/2^{a - 1}$ is strictly log convex in $a$ and takes value $1$ when $a\in\{1,2\}$. 
So, 
\begin{equation}\label{eq:C_a<}
\text{for $a>1$ we have $C_a < 0$ iff $a < 2$, and  $C_a=0$ iff $a=2$. }	
\end{equation}

Consider now the case $a\in(2,3)$. Then, by \eqref{eq:rho'',g}, $\rho$ is strictly concave-convex and hence, by \eqref{eq:rho to 1,g} and Lemma~\ref{lem:rho''}, $\rho$ is increasing-decreasing,  
on $(0,\infty)$. 
Also, $f(\infty-)\break
=g(\infty-)=0$. 
Therefore, by part~(ii) of Proposition~A, $r=g/f$ is increasing-decreasing on $(0,\infty)$. 
Also, by the exactness of the bound $g_a(x)$ at $x=0$ and $x=\infty$, we have 
$r(0+)=r(\infty-)=1$. 
It follows that $r>1$ on $(0,\infty)$, which means that the second inequality in \eqref{eq:g bounds} holds. 

The first inequality in \eqref{eq:g bounds} is proved quite similarly. 


This completes the proof of Proposition~\ref{prop:g}. 
\end{proof}

\begin{proof}[Proof of Proposition~\ref{prop:compar}] 
Take indeed any $a\in(1,2)\cup(2,\infty)$ and any real $x>0$. 
Recall the definition of $C_a$ in \eqref{eq:C_a}. 
Consider \\ 
\begin{equation*}
	d(x):=\Big(\frac{g_a(x)}{G_a(x)} - 1\Big) 2 ((x+b_a)^a - x^a)
	=(2- b_a) x^a + b_a (x+2)^a - 2 (x+b_a)^a + 2b_a C_a, 
\end{equation*}
which equals $g_a(x)-G_a(x)$ in sign, and then 
\begin{equation*}
	d_1(u):=d'(1/u)u^{a-1}/a=2 - 2 (1 + b_a u)^{a-1} + b_a ((1 + 2 u)^{a-1}-1)
\end{equation*}
and 
\begin{equation*}
	\frac{d'_1(u)}{2 (a - 1) b_a}=(1+2 u)^{a-2}-(1+b_a u)^{a-2}<0
\end{equation*}
for $u>0$, since, in view of Lemma~\ref{prop:b}, $b_a-2$ equals $a-2$ in sign. 
So, $d_1$ is decreasing on $(0,\infty)$, from $d_1(0+)=0$. 
It follows that $d_1<0$ and hence $d$ is decreasing on $(0,\infty)$, from $d(0+)=0$. 
Thus, $d<0$ on $(0,\infty)$, 
which completes the proof of Proposition~\ref{prop:compar}. 
\end{proof}

\begin{proof}[Proof of Theorem~\ref{th:a<-1}]
Take indeed any real $a<-1$ and $x>0$. 

Consider first the lower bound $g_a^\lo(x)$ on $\Ga(a,x)$ and, within this consideration, let $g:=g_a^\lo$, for the simplicity of writing. Let then $f:=f_a$, with $f_a$ as in \eqref{eq:f_a}, and let $r=g/f$ and $\rho=g'/f'$, as in Proposition~A. Then 
\begin{equation*}
	\rho(x)=
	\frac{u^4+2 (a-1) u^2-4 a u-a^2}{(u^2+a)^2}, \quad\text{with}\quad u:=x-a,
\end{equation*}
so that $u>-a>1$ and hence $u^2>u>-a$ and $u^2+a>0$. 
Next, 
\begin{equation*}
	\rho'(x)(u^2+a)^3/(4x)
	=x^2 - a (a + 1)  
\end{equation*}
and $a (a + 1)>0$. So, $\rho$ is decreasing-increasing on $(0,\infty)$. 
Also, $f(\infty-)=g(\infty-)=0$. 
Therefore, by part~(ii) of Proposition~A, $r=g/f$ is decreasing-increasing on $(0,\infty)$. 
Also, $\rho(0+)=1=\rho(\infty-)$, $f(0+)=\infty=g(0+)$, and $f(\infty-)=0=g(\infty-)$, 
whence, by the l'Hospital rule for limits, $r(0+)=1=r(\infty-)$. 
Thus, $r<1$ on $(0,\infty)$. We conclude that the first inequality 
in \eqref{eq:bnds,a<-1} and the 
exactness properties 
concerning the lower bound $g_a^\lo(x)$ on $\Ga(a,x)$ do hold. 

The corresponding proof for the upper bound $g_a^\hi(x)$ is similar and even simpler. Indeed, letting here $g:=g_a^\hi$ and, as before, $f:=f_a$, $r=g/f$, and $\rho=g'/f'$, we have 
\begin{equation*}
	\rho(x)=\frac{(x - a)^2 + x}{(x - a)^2} \quad\text{and}\quad  
	\rho'(x)=-\frac{x+a}{(x - a)^3}, 
\end{equation*}
so that $\rho$ is increasing-decreasing on $(0,\infty)$. 
Also, $f(\infty-)=g(\infty-)=0$. 
Therefore, by part~(ii) of Proposition~A, $r=g/f$ is increasing-decreasing on $(0,\infty)$. 
Also, $\rho(0+)=1=\rho(\infty-)$, $f(0+)=\infty=g(0+)$, and $f(\infty-)=0=g(\infty-)$, 
whence, by the l'Hospital rule for limits, $r(0+)=1=r(\infty-)$. 
Thus, $r>1$ on $(0,\infty)$. We conclude that the second inequality 
in \eqref{eq:bnds,a<-1} and the 
exactness properties 
concerning the upper bound $g_a^\hi(x)$ on $\Ga(a,x)$ do hold. 

Theorem~\ref{th:a<-1} is now proved. 
\end{proof}

\begin{proof}[Proof of Proposition~\ref{prop:a->-infty,unif}]
The ratio in \eqref{eq:max ghi/glo} equals $\frac{(x-a)^2+a}{(x-a-1)(x-a)}$, and its partial \break 
derivative in $x$ equals $a(a+1) - x^2$ in sign. So, for each $a<-1$ this ratio attains its maximum in $x>0$ at $x=\sqrt{a(a+1)}$, and the value of this maximum is $\frac{2}{1+\sqrt{1+1/a}}\to1$ as $a\to-\infty$. Thus, asymptotic relations \eqref{eq:max ghi/glo} and \eqref{eq:a->-infty,unif} are verified. 
\end{proof}

\begin{proof}[Proof of Proposition~\ref{prop:a->infty}]
Let indeed $a\to\infty$. 
By Stirling's formula, $b_a\sim a/e$ and hence $G_a(a)=(a/e)^a(1+1/e)^a e^{o(a)}$, whereas $\Ga(a,a)\le\Ga(a)=(a/e)^a e^{o(a)}$, so that 
\begin{equation*}
	\max_{x>0}\dfrac{G_a(x)}{\Ga(a,x)}\ge\dfrac{G_a(a)}{\Ga(a,a)}\ge(1+1/e)^a e^{o(a)}\to\infty, 
\end{equation*}
and the first asymptotic relation in \eqref{eq:max ghi/glo,infty} follows. 

Letting now $w(t):=w_a(t):=(a-1)\ln t-t$, we have $w(a-1)=(a-1)\ln\frac{a-1}e$, $w'(a-1)=0$, and $w''(t)=-\frac{a-1}{t^2}>-\frac1{a-1}$ for $t>a-1$, so that 
\begin{multline*}
	\Ga(a,a-1)=\int_{a-1}^\infty e^{w(t)}dt
	>\int_{a-1}^\infty \exp\Big\{w(a-1)-\frac{(t-(a-1))^2}{2(a-1)}\Big\}dt \\ 
	=\Big(\frac{a-1}e\Big)^{a-1}\sqrt{\frac{\pi(a-1)}2}. 
\end{multline*}
On the other hand, it is easy to see that 
\begin{equation*}
	g_a(a-1)\sim\frac{e^2-1}2\,\Big(\frac{a-1}e\Big)^{a-1}. 
\end{equation*}
So, 
\begin{equation*}
	\max_{x>0}\dfrac{\Ga(a,x)}{g_a(x)}\ge\dfrac{\Ga(a,a-1)}{g_a(a-1)}
	=\frac{2+o(1)}{e^2-1}\sqrt{\frac{\pi a}2}\to\infty, 
\end{equation*}
and the second asymptotic relation in \eqref{eq:max ghi/glo,infty} follows as well. 
\end{proof}

\begin{proof}[Proof of Proposition~\ref{prop:bounded}] 
Since $\Ga(a,x)$, $G_a(x)$, and $g_a(x)$ are (strictly) positive and continuous in $(a,x)$, it follows that $\de G_a(x)$ and $\de g_a(x)$ are bounded away from $-1$ and $\infty$ over all $a\in[1,a_*]$ and $x\in[0,a]$. 

On the other hand, by the mean value theorem and Lemma~\ref{prop:b}, 
\begin{equation*}
	x^{a-1}e^{-x}\le G_a(x)\le(x+b_a)^{a-1}e^{-x}\le x^{a-1}e^{-x}(1+b_a/a)^{a-1}
	\le x^{a-1}e^{-x}(1+b_{a_*})^{a_*-1}   
\end{equation*}
for all $a\in[1,a_*]$ and $x\in[a,\infty)$. 

Again by the mean value theorem and in view of 
\eqref{eq:C_a<}, 
\begin{equation*}
	g_a(x)\ge(x^{a-1}+C_a/a)e^{-x}\ge x^{a-1}e^{-x}  
\end{equation*}
for all real $a\ge2$ and $x>0$.
Also, once again by the mean value theorem,
\begin{equation*}
	g_a(x)\le((x+2)^{a-1}+\Ga(a))e^{-x}\le x^{a-1}e^{-x}((1+2/a)^{a-1}+\Ga(a)/a^{a-1})
	\le x^{a-1}e^{-x} K 
\end{equation*}
for some universal positive real constant $K$ and all $a\in[2,a_*]$ and $x\in[a,\infty)$. 

Finally, by Proposition~\ref{prop:f_a} (to be proved next), 
\begin{equation*}
	x^{a-1}e^{-x}\le\Ga(a,x)\le a x^{a-1}e^{-x}
\end{equation*}
for all $a\in[1,\infty)$ and $x\in[a,\infty)$.

Collecting all the pieces together, we complete the proof of  Proposition~\ref{prop:bounded}.  
\end{proof}

\begin{proof}[Proof of Proposition~\ref{prop:f_a}] 
For all real $x>0$ 
\begin{equation*}
	\Ga(a,x)=\int_x^\infty t^{a-1} e^{-t}\,dt\ge x^{a-1}\int_x^\infty e^{-t}\,dt=x^{a-1} e^{-x}, 
\end{equation*}
which proves the first inequality in Proposition~\ref{prop:f_a}. 

Next, 
for $t>x>a-1$, let $h(t):=(a-1)\ln t-t$. Since the function $h$ is strictly concave, for $t>x$ we have 
\begin{equation*}
	h(t)<h_x(t):=h(x)+h'(x)(t-x)=(a-1)\ln x-x+((a-1)/x-1)(t-x) 
\end{equation*}
So,
\begin{equation*}
	\Ga(a,x)=\int_x^{\infty}e^{h(t)}\,dt
	<\int_x^{\infty}e^{h_x(t)}\,dt =\frac{x^{a-1} e^{-x}}{1-(a-1)/x}, 
\end{equation*}
which proves the second inequality in Proposition~\ref{prop:f_a} as well. 
\end{proof}

\begin{proof}[Proof of Proposition~\ref{prop:g_{a;2}^lo}] 
Inequality \eqref{eq:>g_{a;2}^lo} follows immediately from the discussion of forward-shift bounds in the paragraph containing formula \eqref{eq:B} and the first inequality in \eqref{eq:bnds,a<-1}. The exactness of the bound $g_{a;2}^\lo(x)$ on $\Ga(a,x)$ at $x=\infty$ follows immediately from \eqref{eq:g_{a;2}^lo} and 
Lemma~\ref{lem:infty}. 
\end{proof}

\begin{proof}[Proof of Proposition~\ref{prop:G_{a;-1}}] 
Consider first the case $a\ne0$, so that $a\in(-2,1)\setminus\{0\}$. Then 
inequality \eqref{eq:ineq,G_{a;-1}} follows immediately from 
equality \eqref{eq:eq,G_{a;-1}} and 
inequalities 
\eqref{eq:-1 le a<1} and \eqref{eq:1<a<2}, and 
the exactness of the bound $G_{a;-1}(x)$ on $\Ga(a,x)$ at $x=0$ follows by \eqref{eq:eq,G_{a;-1}}, \eqref{eq:g}, and Lemma~\ref{lem:0}. 

To complete the proof of Proposition~\ref{prop:G_{a;-1}}, consider now the case $a=0$. 
The second equality in \eqref{eq:0,G_{a;-1}} can be obtained as follows. In view of \eqref{eq:eq,G_{a;-1}} and \eqref{eq:g}, write $G_{a;-1}(x)$ as the ratio with denominator $a(a+1)b_{a+1}$; replace $(a+1)b_{a+1}$ in the denominator by $\lim_{a\to0}(a+1)b_{a+1}=b_1$; finally, use l'Hospital's rule. 

The exactness of the bound $G_{0;-1}(x)$ on $\Ga(0,x)$ at $x=0$ follows by \eqref{eq:0,G_{a;-1}} and Lemma~\ref{lem:0}. 
The non-strict version of inequality \eqref{eq:ineq,G_{a;-1}} for $a=0$ follows by continuity from inequality \eqref{eq:ineq,G_{a;-1}} for $a\ne0$. 

However, the strict inequality \eqref{eq:ineq,G_{a;-1}} for $a=0$ requires proof, which is somewhat similar to the proofs of inequalities in Theorems~\ref{th:} and \ref{th:a<-1}. 
Again, we are going to apply Proposition~A, now to the functions $f=\Ga(0,\cdot)$ and $g=G_{0;-1}$, where $G_{0;-1}$ is as in \eqref{eq:0,G_{a;-1}}.  
Then for real $x>0$ 
\begin{align*}
	\rho(x)&:=\frac{g'(x)}{f'(x)}
	=x \Big(1+\frac{x-1}{b_1}\Big) \ln\frac{b_1+x}x-x+1, \\ 
	\rho'(x)&=\frac{1}{b_1+x}-2
	+\frac{b_1+2x-1}{b_1}\,\ln\frac{b_1+x}x,\\
	\rho''(x)&=\frac{d_2(x)}{b_1 x (b_1+x)^2}, \\ 
	d_2(x)&:=
	2 x (b_1+x){}^2 \ln\frac{b_1+x}{x}-b_1 (b_1 (3 x-1)+b_1^2+2
   x^2), \\ 
	\rho'''(x)&=b_1\,\frac{(b_1-1) b_1-x (3-b_1)}{x^2 (b_1+x)^2}. 
\end{align*}

By Lemma~\ref{prop:b}, $1<b_1<2$. So, $\rho'''$ is $+-$ on $(0,\infty)$ -- that is, 
there is some $c\in[0,\infty]$ such that $\rho'''>0$ on $(0,c)$ and $\rho'''<0$ on $(c,\infty)$ (in this case, we actually have $c\in(0,\infty)$). 
So, $\rho''$ is increasing-decreasing on $(0,\infty)$. Also, $x^3\rho''(x)\to(1-b_1/3)b_1>0$ as $x\to\infty$. 
So, $\rho''$ is $-+$ on $(0,\infty)$.  
So, $\rho$ is strictly concave-convex on $(0,\infty)$. 
Also, $\rho(\infty-)=b_1/2\in\R$. 
So, by Lemma~\ref{lem:rho''}, $\rho$ is increasing-decreasing on $(0,\infty)$.  
Also, $f(\infty-)=g(\infty-)=0$. 
Therefore, by part~(ii) of Proposition~A, $r=g/f$ is increasing-decreasing on $(0,\infty)$. 

Also, for real $x>0$ 
\begin{align*}
	r'(x)&=\frac{d(x)}{b_1 e^{2 x} x \Ga(0, x)^2}, \\ 
	d(x)&:=
	(b_1+x) (\ln (b_1+x)-\ln x)-b_1 \\
	&+e^x \Ga(0,x) 
	(b_1(x-1)+x (1-b_1-x) (\ln (b_1+x)-\ln x)). 
\end{align*}
Making here the substitutions $\ln(b_1+x)=\ln b_1 + c_1 x$, $e^x=1 + c_2 x$, and, in accordance with \eqref{eq:Ga,a=0} and \eqref{eq:b}, 
$\Ga(0,x)=-\ln x-\ga+c_3 x=-\ln x+\ln b_1-1+c_3 x$, where $c_j=c_j(x)=O(1)$ as $x\downarrow0$ for each $j\in\{1,2,3\}$, we see that $r'(0+)=1/b_1-1<0$. 
So, the increasing-decreasing function $r$ is actually decreasing everywhere on $(0,\infty)$. 
Also,  
the already established exactness of the bound $G_{0;-1}(x)$ on $\Ga(0,x)$ at $x=0$ means 
that $r(0+)=1$. 
Thus, $r<1$ on $(0,\infty)$; that is, inequality \eqref{eq:ineq,G_{a;-1}} holds for $a=0$. 



This completes the proof of Proposition~\ref{prop:G_{a;-1}}.  
\end{proof}

\begin{proof}[Proof of Proposition~\ref{prop:G_{a;1}}] 
This follows immediately from Proposition~\ref{prop:taming} (with $k=1$) and Theorem~\ref{th:}. 
More specifically, 
the case $1<a<2$ follows from part~(ii) of Proposition~\ref{prop:taming} and \eqref{eq:-1 le a<1}; 
the case $a=2$ follows from parts~(i) and (ii) of Proposition~\ref{prop:taming} and \eqref{eq:a=1}; and 
the case $2<a<3$ follows from part~(i) of Proposition~\ref{prop:taming} and \eqref{eq:1<a<2}. 
\end{proof}

\begin{proof}[Proof of Proposition~\ref{prop:g_{a;1}^lo}] 
The first inequality in \eqref{eq:ineq,g_{a;1}^lo} follows immediately from the first equality in \eqref{eq:eq,g_{a;1}^lo}, the first inequality in \eqref{eq:bnds,a<-1}, and identity \eqref{eq:Ga,fwd}. 
The second inequality in \eqref{eq:ineq,g_{a;1}^lo} follows because 
\begin{equation*}
	g_a^\hi(x)-g_{a;1}^\lo(x)=\frac{x^{1 + a}e^{-x}}{(x-a)\big((x-a)^2-a+2 x\big)}>0
\end{equation*}
for $a<0$ and $x>0$. 

The exactness of the upper bound $g_{a;1}^\lo(x)$ on $\Ga(a,x)$ at $x=0$ and at $x=\infty$ follows immediately from inequalities \eqref{eq:ineq,g_{a;1}^lo} and the exactness of $g_a^\hi(x)$ at $x=0$ and at $x=\infty$. 
%
\end{proof}

\bibliographystyle{abbrv}

\bibliography{P:/pCloudSync/mtu_pCloud_02-02-17/bib_files/citations10.13.18a}

%
%
%
%
%
%
%
%
%
%

\end{document}